\documentclass[11pt]{amsart}

\usepackage{amsmath,amssymb,latexsym,soul,cite,mathrsfs}

\usepackage{color,enumitem,graphicx}
\usepackage[colorlinks=true,urlcolor=blue,
citecolor=red,linkcolor=blue,linktocpage,pdfpagelabels,
bookmarksnumbered,bookmarksopen]{hyperref}
\usepackage[english]{babel}

\usepackage[left=2.9cm,right=2.9cm,top=2.8cm,bottom=2.8cm]{geometry}
\usepackage[hyperpageref]{backref}

\usepackage[colorinlistoftodos]{todonotes}
\makeatletter
\providecommand\@dotsep{5}
\def\listtodoname{List of Todos}
\def\listoftodos{\@starttoc{tdo}\listtodoname}
\makeatother

\numberwithin{equation}{section}

\def\cal{\mathcal}

\newtheorem{theorem}{Theorem}[section]

\newtheorem{lemma}[theorem]{Lemma}

\newcommand\R{\mathbb R}
\newcommand\N{\mathbb N}

\begin{document}

\title[ Ground state  for fractional  equations with critical nonlinearity]
{Ground state solutions for fractional scalar field equations under a general critical nonlinearity}

\author{Claudianor O. Alves}
\author{Giovany M. Figueiredo}
\author{Gaetano Siciliano}

\address[Claudianor O. Alves]{\newline\indent Departamento de Matem\'atica
\newline\indent 
Universidade Federal de Campina Grande,
\newline\indent
58429-970, Campina Grande - PB - Brazil}
\email{\href{mailto:coalves@dme.ufcg.edu.br}{coalves@dme.ufcg.edu.br}}

\address[Giovany M. Figueiredo]
{\newline\indent Faculdade de Matem\'atica
\newline\indent 
Universidade Federal do Par\'a
\newline\indent
66075-110, Bel\'em - PA, Brazil}
\email{\href{mailto:giovany@ufpa.br}{giovany@ufpa.br}}

\address[Gaetano Siciliano]{\newline\indent Departamento de Matem\'atica
\newline\indent 
Instituto de Matem\'atica e Estat\'istica
\newline\indent 
 Universidade de S\~ao Paulo 
\newline\indent 
Rua do Mat\~ao 1010,  05508-090 S\~ao Paulo, SP, Brazil }
\email{\href{mailto:sicilian@ime.usp.br}{sicilian@ime.usp.br}}


%

\pretolerance10000


\begin{abstract}
\noindent In this paper we study existence of ground state solution to the following problem
$$
(- \Delta)^{\alpha}u   = g(u) \ \ \mbox{in} \ \ \mathbb{R}^{N}, \ \ u \in H^{\alpha}(\mathbb R^N)
$$  
where $(-\Delta)^{\alpha}$ is the fractional Laplacian, $\alpha\in (0,1)$. 
We treat both cases $N\geq2$ and $N=1$ with $\alpha=1/2$.
The function $g$ is a general nonlinearity of Berestycki-Lions type
which is allowed to have critical growth: polynomial in case $N\geq2$, exponential if $N=1$.
%
\end{abstract}

\thanks{Claudianor Alves was partially supported by CNPq/Brazil Proc. 304036/2013-7 ; Giovany M. Figueiredo was partially
supported by  CNPq, Brazil; Gaetano Siciliano  was partially supported by
Fapesp and CNPq, Brazil. }

\maketitle

\section{Introduction}

In the present paper, we are interesting in the existence of ground state solution for a class of nonlocal problem of the following type 
\begin{equation}\tag{P}\label{eq:P}
(- \Delta)^{\alpha}u   = g(u), \quad \mbox{in} \quad \mathbb{R}^N
\end{equation}
where $N \geq 1$, $\alpha \in (0,1),$ $(- \Delta)^{\alpha}$ denotes the fractional Laplacian operator and $g$ is a $C^{1}-$function verifying some conditions which will be mentioned later on. 

\medskip

The main motivation for this paper comes from the papers Berestycki and Lions \cite{BL} and Berestycki, Gallouet and Kavian \cite{BGK} which have studied the existence of solution for \eqref{eq:P} in the local case $\alpha =1$, that is, for a class of elliptic equations like 
\begin{equation}\label{LE1*}
- \Delta u   = g(u), \quad \mbox{in} \quad \mathbb{R}^N, 
\end{equation}
where $N \geq 2$, $\Delta$ denotes the Laplacian operator and $g$ is a continuous function verifying some conditions. In \cite{BL}, Berestycki  and Lions have assumed $N \geq 3$ and the following conditions on $g$:
$$
- \infty < \liminf_{s \to 0^+}\frac{g(s)}{s} \leq \limsup_{s \to 0^+}\frac{g(s)}{s}\leq -m<0, 
$$
$$
\limsup_{s \to 0^+}\frac{g(s)}{s^{2^{*}-1}}\leq 0, 
$$
$$
\mbox{there is} \quad \xi>0 \, \, \mbox{such that} \,\, G(\xi)>0, 
$$
where $G(s)=\int_{0}^{s}g(t)\,dt$.

In \cite{BGK}, Berestycki, Gallouet and Kavian have studied the case where $N=2$ and the nonlinearity $g$ possesses an exponential growth of the type
$$
\limsup_{s \to 0^+}\frac{g(s)}{e^{\beta s^2}}=0, \quad \forall \beta >0.
$$

In the two papers above mentioned, the authors have used the variational method to prove the existence of solution for \eqref{eq:P}. The main idea is to solve the minimization problem 
$$
\min \left\{\frac{1}{2}\int_{\mathbb{R}^N}|\nabla u|^{2} \,dx \,:\, \int_{\mathbb{R}^N}G(u)\,dx=1  \right\}  
$$
and 
$$
\min \left\{\frac{1}{2}\int_{\mathbb{R}^N}|\nabla u|^{2}\,dx \,:\, \int_{\mathbb{R}^N}G(u)\,dx=0  \right\}  
$$
for $N \geq 3$ and $N=2$ respectively. After that, the authors showed that the minimizer functions of the above problem are in fact ground state solutions of (\ref{LE1*}). By a ground state solution, we mean a solution $u \in H^{1}(\mathbb{R}^N)$ which satisfies
$$
E(u) \leq E(v) \quad \mbox{for all nontrival solution} \  v  \ \text{of} \ (\ref{LE1*}),
$$ 
where $E:H^{1}(\mathbb{R}^N) \to \mathbb{R}$ is the energy functional associated to \eqref{eq:P} given by
$$
E(u)=\frac{1}{2}\int_{\mathbb{R}^N}|\nabla u|^{2}\,dx - \int_{\mathbb{R}^N}G(u)\,dx.
$$

After, Jeanjean and Tanaka in \cite{JJTan} showed that the mountain pass level of  $E$ is a critical level and it is 
indeed the lowest critical level.

In the above mentioned papers, the nonlinearity does not have critical growth. Motivated by this fact,  Alves, Montenegro and  Souto in \cite{AlvesSoutoMontenegro} have studied the existence of ground state solution for 
\eqref{eq:P} by supposing that $g(s)=f(s)-s$ and that $f$ may have critical growth, more precisely, the following condition were considered:    
\begin{equation*}
\lim_{s \to 0}\frac{f(s)}{s}=0
\end{equation*}
\begin{equation*}
\limsup_{s \to +\infty}\frac{f(s)}{s^{2^{*}-1}}\leq 1, \,\, \mbox{if} \,\, N \geq 3 
\end{equation*}
\begin{equation*}
\lim_{s \to +\infty}\frac{f(s)}{e^{\beta s^{2}}}=0 \,\, (\beta < \beta_0) \,\, \mbox{if} \,\, \beta > \beta_0 \,\, ( \beta < \beta_0) \,\, \mbox{when} \,\, N=2
\end{equation*}
\begin{equation*}
H(s)=f(s)s-2F(s) \geq  0 \quad \forall s > 0 \quad \mbox{where} \quad F(s)=\int_{0}^{s}f(t)\,dt,
\end{equation*}
\noindent there is $\tau >0$ and $q \in (2,2^{*})$ if $N \geq 3$ and $q \in (2, +\infty)$ if $N=2$ such that
\begin{equation*}
f(s) \geq \tau s^{q-1}, \quad \forall s \geq 0. 
\end{equation*}

By using the variational method, the authors in  \cite{AlvesSoutoMontenegro} give
 a unified approach in order to deal with subcritical and critical case. However, we would like to point out that the Concentration Compactness Principle of Lions \cite{lionsI} was crucial for the case $N \geq 3$. For the case $N=2$, as in the previous references, a Trudinger-Moser inequality due to Cao \cite{Cao} was the main tool used. A similar study was made for the critical case and $N \geq 3$ in Zhang  and Zou \cite{ZZ}.

\medskip

After a review bibliographic, we have observed that there is no a version of the paper \cite{AlvesSoutoMontenegro} for the fractional Laplacian operator.  Motivated by this fact, we have decide to study this class of problem. However, we would like point out that some estimates made in \cite{AlvesSoutoMontenegro} are not immediate for fractional Laplacian operator. For example, there is some restriction to use Concentration Compactness Principle of Lions \cite{lionsI} as  mentioned in Palatucci and Pisante \cite[Theorem 1.5]{PalatucciPisante} for the dimension $ N\geq 2$ and $\alpha \in (0,1)$. To overcome this difficulty, we use a new approach which do not use the Concentration Compactness Principle of \cite{PalatucciPisante}. For the dimension $N=1$ and $\alpha =1/2$, we  use a Trudinger-Moser inequality due to Ozawa \cite{Ozawa} which also permits to apply variational methods in this case. Here, it is very important to mention that Zhan, do \'O and Squassina \cite[Theorem 4.1]{ZOS} studied the existence of ground state solution for \eqref{eq:P} for $N \geq 2$, by supposing that $g$ satisfies  
$$
\lim_{s \to +\infty}\frac{g(s)}{s^{2^{*}_\alpha -1}}=b>0.
$$  
These condition is not assumed in our paper, and so, our results complete the studied made in that paper.

Before stating our main results, we must fix some notations. We will look for weak solutions of \eqref{eq:P}
hence the natural setting  involves the fractional Sobolev spaces $H^\alpha(\mathbb R^N)$ defined as
\[
H^\alpha(\mathbb R^N)=\big\{u\in L^2(\mathbb R^N):\,
(-\Delta)^{\alpha/2}u \in L^2(\mathbb R^N)\big\}
\]
endowed with scalar product and (squared) norm given by
\[
(u,v)= \int_{\mathbb R^{N}} (-\Delta)^{\alpha/2}u (-\Delta)^{\alpha/2}v \, dx+ \int_{\mathbb R^{N}} uv \, dx,
\qquad
\|u\|^2=|(-\Delta)^{\alpha/2}u |_2^2+ |u|_2^2.
\]
It is well known that $H^\alpha(\mathbb R^N)$ is a Hilbert space with the above scalar product. 
We are denoting with $|u|_{p} = (\int_{\mathbb R^{N}}|u|^{p}dx)^{1/p}$ the $L^{p}-$norm of $u$, and by $(-\Delta)^\alpha$ the fractional Laplacian, which is the pseudodifferential operator defined via the Fourier transform of the following way
$$
\mathcal F((-\Delta)^{\alpha}u)=|\cdot|^{2\alpha}\mathcal Fu.
$$
It is known that  $H^{\alpha}(\mathbb R^{N})$ has continuos embedding into $L^{q}(\mathbb R^{N})$
for suitable $q$ depending on $N$: we will denote by $C_{q}>0$ the embedding constant.

It is useful to introduce  also the  homogeneous fractional Sobolev space
$$ \mathcal{D}^{\alpha, 2}(\R^N) = \Big\{u \in L^{2^{*}_{\alpha}}(\R^N):     \ \int_{\R^N} |(-\Delta)^{\alpha/2}u |^{2}  \, dx < \infty  \Big\}
$$ 
where hereafter $2^{*}_{\alpha}=\frac{2N}{N-2 \alpha}$ for $N \geq 2$.  It is well known that the following inequality holds 
\begin{equation} \label{definicaoS}
S \, \left( \int_{\R^N} |u|^{2^{*}_{\alpha}} \, dx   \right)^{2/2^{*}_{\alpha}} \leq  \int_{\R^N} |(-\Delta)^{\alpha/2}u |^{2} \, dx  \quad\hbox{for all $u \in \mathcal{D}^{\alpha, 2}(\R^N)$}   
\end{equation}
for some positive $S>0$. For these facts and the relation between the fractional Laplacian and the fractional Sobolev space $H^{\alpha}(\mathbb R^{N})$, we refer the reader to classical books on Sobolev space, and to the monograph  \cite{DPV}.

\medskip

We will study \eqref{eq:P} by variational methods: 
its solutions will be found as  critical points of a $C^{1}$ functional 
$I:H^{\alpha}(\mathbb R^{N}) \to \mathbb R$. 
Actually our results concern the existence of {\sl ground state solutions}, 
that is a solution $u \in H^\alpha(\R^N)$ such that $I(u) \leq I(v)$ for every nontrivial solution $v \in H^\alpha(\R^N)$ of \eqref{eq:P}.  In view of this, we make the following assumptions on the nonlinearity $f$.
More precisely 
we assume that
 $f:\mathbb R\to \mathbb R$ is a $C^{1}$-function satisfying  

\medskip

\begin{enumerate}[label=(f\arabic*),ref=f\arabic*,start=1]
\item\label{f_{1}} $\displaystyle\lim_{s \rightarrow 0^+} f(s)/ s =0$; \smallskip
\item \label{f_{2}} $\displaystyle \limsup_{s \rightarrow + \infty} f(s) / s^{2^{*}_\alpha -1} \leq 1$; \smallskip
\item\label{f_{3}} $ f(s) s - 2F(s) \geq 0$ for $s > 0$, 
where $F(s)= \int_0^s f( t) \, dt$; \medskip 
\item\label{f_{4}} 
$f(s) \geq \tau s^{q-1}$,  $s \in \R$ with $s \geq 0$, where \medskip
\begin{itemize}
\item If $N\geq2$, we assume $q \in (2, 2^{*}_{\alpha})$ and  \medskip
$$\tau>\tau^{*}:=\left[ 2^{(2\alpha-N)/2\alpha} S^{-N/2\alpha} \frac N\alpha \left(\frac{2N}{N-2\alpha}\right)^{(N-2\alpha)/2\alpha}\right]^{(q-2)/2} 
\left( \frac{q-2}{2q} \right)^{(q-2)/2} C_q^{q/2},$$
\medskip

\item If $N=1$, we assume $q>2$ and 
$$
\tau>\tau^* = \left( \frac{q-2}{q} \right)^{(q-2)/2} C_q^{q/2};
$$
\end{itemize}
\medskip
\item\label{f_{5}} there exist $\omega \in (0,\pi)$ and 
$\beta_0 \in (0, \omega],$
such that
$$
\lim_{s\rightarrow  +\infty}\frac{f(s)}{e^{\beta s^2}}=0, \ \forall \beta > \beta_0,\quad \mbox{and} \quad  
\lim_{s \rightarrow +\infty}\frac{f(s)}{e^{\beta s^2}}=+\infty, \ \forall \beta < \beta_0.
$$
\end{enumerate}

As we can see, a critical growth for the function $f$ is allowed.
Note also that a weaker condition than the usual Ambrosetti-Rabinowitz condition is imposed on $f$, see condition 
\eqref{f_{3}}.

\medskip

Our main results are the following one.

\begin{theorem}\label{groundN3}
Suppose  that  $N \geq 2$  and $f$ satisfies   \eqref{f_{1}}-\eqref{f_{4}}. Then
 problem  \eqref{eq:P} admits a ground state solution which is non-negative, radially symmetric and decreasing.
 
\end{theorem}


\begin{theorem}\label{groundN1}
Suppose  that  $N = 1$  and $f$ satisfies \eqref{f_{1}}, \eqref{f_{4}} and \eqref{f_{5}}. Then
problem  \eqref{eq:P} admits a ground state solution
a ground state solution which is non-negative, radially symmetric and decreasing.
%
\end{theorem}


\medskip

The plan of the paper is the following: In Section \ref{tools} we study the case
$N\geq2$.  We first introduce the variational framework, then give some preliminaries results and Lemmas 
which will be useful to prove Theorem \ref{groundN3}.
In Section \ref{tools1} we consider the case $N=1$ and $\alpha=1/2$, where again, after some preliminaries,
the proof of Theorem \ref{groundN1} is given.

\medskip 

Before concluding this introduction, we would like to cite some papers involving the fractional Laplacian operator where the problem is related to the problem \eqref{eq:P} in some sense, see for example,  Ambrosio \cite{Am}, Barrios, Colorado, de Pablo and S\'anchez \cite{barrios},  Frank and Lenzmann \cite{FL}, Felmer, Quass and Tan \cite{FQT}, Iannizzotto and Squassina \cite{IS}, Zhang, do \'O and Squassina \cite{ZOS} and their references.

\medskip

\noindent {\bf Notations} As a matter of notations, we will use in all the paper the letter $C, \overline C, C', \ldots$  to denote 
suitable positive constants whose exact value is insignificant for our purpose.

\section{The case $N\geq2$}\label{tools}
\subsection{The variational framework}



%

\noindent The energy functional $I: H^\alpha(\R^N) \rightarrow \R$ associated to equation \eqref{eq:P}
 is defined as follows 
$$I(u) = \frac 12 \| u \|^2 -  \int_{\R^N} F(u) \, dx.$$
Under assumptions  \eqref{f_{1}} and \eqref{f_{2}}, $I \in C^1(H^{\alpha}(\R^N), \R)$  with  Frech\'et derivative given by 
$$
I'(u)[v] = \displaystyle\int_{\mathbb{R}^{N}}(-\Delta)^{\alpha/2}u(-\Delta)^{\alpha/2}v dx+ \displaystyle\int_{\mathbb{R}^{N}} uv dx - 
\int_{\R^N} f(u) v dx, \quad \forall u,v \in H^\alpha(\R^N).
$$
Hence the critical points are easily seen to be  weak solutions to \eqref{eq:P}.

\medskip

We remark two inequalities which will be frequently used in the sequel.
From \eqref{f_{1}} and \eqref{f_{2}}, for any $\varepsilon > 0$ there exists $C_{\varepsilon} > 0$ such that 
\begin{equation} \label{ESTIMATIVA1}
|f(s) | \leq \varepsilon |s|+ C_{\varepsilon} |s|^{2^{*}_{\alpha} -1} \quad \hbox{for all $s >0$}
\end{equation}
and, then by integration, 
\begin{equation} \label{ESTIMATIVA2}
 |F(s) | \leq \frac{\varepsilon}{2} s^{2} + C_{\varepsilon} |s|^{2^{*}_{\alpha} } \quad \hbox{for all $s >0$}.
\end{equation}

\medskip


\medskip


Once we intend to find nonnegative solution, we will assume that $f(s) =0$ for every $s \leq 0$. 
Let us consider the set of non-zero critical points of $I$, that is non trivial solution of \eqref{eq:P},
$$
\Sigma = \{  u \in H^\alpha(\R^N) \setminus \{0 \}: I'(u) =0 \}, 
$$  
and define
$$  
m = \inf_{u \in \Sigma} I(u)
$$
the so called {\sl ground state level}. 

Now, denoting with   $G(u) =  F(u)  -\displaystyle\frac{u^2}{2}$, the primitive of
$g(u) = f(u)  -u$,  let us introduce the set
\begin{equation}\label{eq:M}
\mathcal{M}=  \left\{ u \in H^\alpha(\R^N) \setminus \{0 \}: \int_{\R^N} G(u) \, dx =1 \right \} 
\end{equation} 
and 
\begin{equation}\label{eq:D}
T(u) =\frac 12 \int_{\R^N}  |(-\Delta)^{\alpha/2}u|^2  dx,\qquad D =  \inf_{u\in \mathcal M}  T(u). 
\end{equation}

In particular
$$ 2 D = \inf_{u \in \mathcal{M}} \Big\{  \int_{\R^N}  |(-\Delta)^{\alpha/2}u|^2  dx\Big\}. 
$$ 
It is worth to point out that if we define the $C^{1}$ functional
\begin{equation*}\label{eq:J}
J(u) := \int_{\mathbb R^{N}} G(u)dx -1,
\end{equation*}
it holds   from \eqref{f_{3}}:
\begin{equation}\label{eq:positividade}
u\in \mathcal M \Longrightarrow J'(u)[u]=\int_{\mathbb R^{N}}(f(u)u - u^{2}) dx =
 \int_{\mathbb R^{N}} (f(u) u - 2F(u))dx +2\int_{\mathbb R^{N}}G(u)\geq 2.
\end{equation}
The last information will be used later on.

In addition, we define the min-max level associated to the functional $I$
\begin{equation}\label{eq:b}
b = \inf_{\gamma \in \Gamma} \max_{t \in [0, 1]} I(\gamma(t))
\end{equation}
where 
$$ \Gamma = \{\gamma \in C\left([0, 1], H^{\alpha}(\R^N) \right): \gamma(0) =0 \ \hbox{ and } \ I(\gamma(1)) < 0\}
$$
which is not empty since $I$ has a Mountain Pass Geometry.

Let us define also the set, usually called {\sl Pohozaev manifold},
\begin{equation*}
\mathcal{P} = \left\{u \in H^\alpha(\R^N, \R) \setminus \{ 0 \}: \frac{N-2\alpha}{2} \int_{\R^N}  |(-\Delta)^{\alpha/2}u|^2  \, dx = N  \int_{\R^N} G(u)  \, dx \right\}.
\end{equation*}
which, according to \cite[Proposition 4.1]{ChangWang}, contains any weak solution of \eqref{eq:P}. If we denote by 
$$
p = \inf_{u \in \mathcal{P}} I(u),
$$ 
from \cite[Lemma 2.4]{LilianeRquelSquassina} it holds that 
\begin{equation}\label{eq:valordep}
p =  \frac {\alpha}{N} \left(\frac{N-2\alpha}{2N} \right)^{(N-2\alpha)/2\alpha} (2D)^{N/2\alpha}.
\end{equation}

\subsection{Some preliminary stuff}

At this point we establish some preliminary  results which will be useful in order to prove Theorem \ref{groundN3}.

\begin{lemma}\label{blowbound}
It holds 
$$\frac {\alpha}{N} \left(\frac{N-2\alpha}{2N} \right)^{(N-2\alpha)/2\alpha} (2D)^{N/2\alpha} \leq b$$
where $b$ is the min-max level of $I$ defined in \eqref{eq:b}.
\end{lemma}

\begin{proof}
 Indeed, from \cite[Lemma 2.3]{LilianeRquelSquassina}, for each $\overline \gamma \in \Gamma$ with 
$$\Gamma = \{\overline{\gamma} \in C\left([0, 1], H^\alpha(\R^N) \right): \overline{\gamma}(0) =0 \ \hbox{and} \ I(\overline{\gamma}(1)) < 0  \} $$ 
it results   $\overline{\gamma}([0, 1]) \cap \mathcal{P} \neq \emptyset$. Then, there exists $t_0 \in [0, 1]$ such that $\overline{\gamma}(t_0) \in \mathcal{P}$. So
$$ 
p \leq I(\overline{\gamma}(t_0)) \leq \max_{t \in [0, 1]} I(\overline{\gamma}(t))
$$
from where it follows that $p \leq  b$ and the result follows from \eqref{eq:valordep}.
 \end{proof}

The next result is standard. We recall the proof for the reader's convenience.
\begin{lemma}\label{Mmanif}
The set  $\mathcal{M}$ defined in \eqref{eq:M} is not empty and a $C^1$ manifold. 
\end{lemma}
\begin{proof}
Observe that, fixed  $0\not\equiv\varphi\in C^{\infty}_{0}(\mathbb R^{N}), \varphi\geq0$ the function
$h(t)=\int_{\mathbb R^{N}}G(t\varphi)dx $ is strictly negative for small $t$ and $h'(t)>0$
for $t$ large; this implies that there exists some $\bar t>0$ such that $\bar t\varphi\in \mathcal M.$
Moreover $\mathcal M$ is a $C^{1}$ manifold in virtue of \eqref{eq:positividade}.
\end{proof}

The next steps consists in proving the boundedness of the minimizing sequences in $H^{\alpha}(\mathbb{R}^N)$ for the problem
\begin{equation}  \label{MINIMIZA1}
\min \left\{\frac{1}{2}\int_{\mathbb{R}^N}|(- \Delta u)^{\frac{\alpha}{2}} \,dx \,:\, \int_{\mathbb{R}^N}G(u)\,dx=1  \right\}. 
\end{equation}

\begin{lemma}\label{bdedminseq}
	Any minimizing sequence $\{ u_n \}\subset \mathcal M$ for $T$  is bounded in $H^\alpha(\R^N)$.
\end{lemma}

\begin{proof}
	Let $\{ u_n \}\subset \mathcal M  $ be a minimizing sequence for $T$, then
	$$T(u_{n}) =\frac 12 \int_{\R^N}  |(-\Delta)^{\alpha/2}u_n|^2  dx \longrightarrow D \quad \hbox{as $n \rightarrow + \infty$} $$ 
	and 
	$$ 
	\int_{\R^N} G(u_n) dx = 1, \quad \hbox{that is}, \quad  \int_{\R^N} \left( F(u_n)   -  \frac 12 u_n^2 \right) dx =1. 
	$$
	Then
	\begin{equation}\label{boundgradA}
	\frac 12\int_{\R^N}   |(-\Delta)^{\alpha/2}u_{n}|^2 \, dx \leq C \quad \hbox{for all $n \in \N$ and for some constant $C > 0$ } 
	\end{equation}
	and  
	$$\int_{\R^N}  F(u_n)  \,dx  = 1 +\frac 12 \int_{\R^N}   u_n ^2  \, dx.$$
	By using (\ref{ESTIMATIVA1}) with $\varepsilon =  1/4$, we get
	$$
	1 +\frac 12 \int_{\R^N}   u_n^2  \, dx \leq  \frac 14 \int_{\R^N}   u_n ^2  \, dx + C_{1/4} \int_{\R^N}   |u_n|^{2^{*}_{\alpha}}  \, dx. 
	$$ 
	Then, for every $n \in \N$, by using   \eqref{boundgradA}, it follows
	$$\frac 12 \int_{\R^N}   u_n^2   dx \leq  C_{1/4} \, \int_{\R^N}   |u_n|^{2^{*}_{\alpha}}  dx 
	\leq C_{1/4}   C  \int_{\R^N}   |(-\Delta)^{\alpha/2}u_{n}|^2   \, dx \leq \overline C. $$ 
	Consequently $\{ u_n \}  $ is bounded also in $L^2(\R^N)$ and this ensures its boundedness in $H^\alpha(\R^N)$.
\end{proof}

By the Ekeland Variational Principle 
we can assume that the minimizing sequence $\{u_{n}\}$ is also a {\sl Palais-Smale sequence}, that is,  there exists a sequence of Lagrange multipliers 
$\{ \lambda_n \} \subset \mathbb R$ such that 
\begin{equation}\label{eq:minimizing}
\frac 12  \int_{\R^N}  |(-\Delta)^{\alpha/2}u_{n}|^2    \, dx  \longrightarrow  D \qquad \hbox{as $n \rightarrow + \infty$} 
\end{equation}
and 
\begin{equation} 
\label{LAMBDA1}
T'(u_n) - \lambda_n J'(u_n) \longrightarrow 0 \  \hbox{in  } (H^\alpha(\R^N))^{-1}\qquad \text{ as } n \rightarrow + \infty.
\end{equation}
In the remaining part of this section,  $\{\lambda_{n}\}$ will be the associated sequence of Lagrange multipliers. At this point it is useful to establish some properties of the levels $D$ and $b$.

\begin{lemma}\label{Dposit}
The number $D$ given by \eqref{eq:D} is positive, namely, $ D > 0$.
\end{lemma}

\begin{proof}
Clearly by definition $D \geq 0$. Suppose, by contradiction, that $D =0$.  If  $\{ u_n \}  $ is a minimizing sequence for $D =0$, then 
$$\frac 12 \int_{\R^N}  |(-\Delta)^{\alpha/2}u_n|^2   dx \to 0 \quad \hbox{as} \quad n \to + \infty 
$$ 
and 
$$ 1 = \int_{\R^N} G(u_n) \, dx =  \int_{\R^N} \left( F(u_n)   - \frac 12 u_n^2 \right)  dx. $$ 
Then, for any $\varepsilon>0$, see (\ref{ESTIMATIVA2}), 
$$1 +  \frac 12\int_{\R^N}  u_n^2  \, dx =  \int_{\R^N} F(u_n)  \, dx \leq \frac{\varepsilon}{2} \int_{\R^N}  u_n^2  \, dx  
+ \frac{C_{\varepsilon}}{2^{*}_{\alpha}} \int_{\R^N}  |u_n|^{2^{*}_{\alpha}}  \, dx  $$
so that
$$ 1+ \frac 12(1 - \varepsilon) \int_{\R^N}  u_n^2  \, dx  \leq C_{\varepsilon} \int_{\R^N}  |u_n|^{2^{*}_{\alpha}}  \, dx \leq C_{\varepsilon} C \int_{\R^N} |(-\Delta)^{\alpha/2}u_{n}|^2  \, dx.
$$
By choosing $\varepsilon = 1/2$, we obtain 
$$
1 \leq C_{ 1/2} C \int_{\R^N}  |(-\Delta)^{\alpha/2}u_n|^2  \, dx \longrightarrow 0 \qquad  \hbox{ as $n \longrightarrow + \infty$}.
$$ 
This contradiction concludes the proof.
\end{proof}

\begin{lemma}\label{lambdanD}
The sequence of Lagrange multipliers $\{ \lambda_n \}$
associated to the minimizing sequence $\{u_{n}\}$  
is bounded. More precisely, we have that  
$$
0<\liminf_{n \to +\infty}\lambda_n \leq \limsup_{n \to + \infty} \lambda_n \leq D.
$$ 
Hence, for some subsequence, still denoted by $\{\lambda_n\}$, we can assume that $\lambda_n \to \lambda^{*}$, for some $\lambda^{*} \in (0, D]$.

\end{lemma}
\begin{proof}
By \eqref{LAMBDA1}, 
\begin{equation}\label{eq:T'J'}
2T(u_{n}) =T'(u_n)[u_n] = \lambda_n J'(u_n)[u_n]  +o_{n}(1).
\end{equation}
Then, from \eqref{eq:positividade}
$$
2T(u_n) \geq 2\lambda_{n} +o_{n}(1) 
$$
which implies, taking into account \eqref{eq:minimizing},
$$ 
\limsup_{n \rightarrow + \infty} \lambda_n \leq \frac 12 \limsup_{n \rightarrow + \infty} \int_{\R^N}  |(-\Delta)^{\alpha/2}u_{n}|^2 dx =D.  
$$ 
Since $\{u_{n}\}$ is a bounded minimizing sequence, it is easy to see that 
$|J'(u_{n})[u_{n}] |= |\int_{\mathbb R^{N}} g(u_{n}) u_{n}| \leq C$, and then
by \eqref{eq:T'J'} and the fact that $2T(u_{n})\to 2D>0$, we infer that 
$$
\liminf_{n \to +\infty}\lambda_n >0.
$$
The proof is thereby completed.
\end{proof}

In the sequel, we will show that  a minimizing sequence  for $(\ref{MINIMIZA1})$ can be choose nonnegative and radially symmetric around the origin. 
Note that for our proof we do not need  to consider the ``odd extension'' of the nonlinearity,
as it is usually done in the literature to show that the minimizing 
sequence can be replaced by the sequence of the absolute values.
 In fact we will prove that 
the minimizing sequence  can be replaced, roughly speaking, with the sequence of the positive parts.

\begin{lemma} \label{positividade}  Any  minimizing sequence $\{u_n\}$  for $(\ref{MINIMIZA1})$
can be assumed
radially symmetric around the origin  and nonnegative.
\end{lemma} 
\begin{proof}To begin with, we recall that $F(s)=0$ for all $s \leq 0$. Thus,
$F(u_n)=F(u_n^{+})$ for all $n \in \mathbb{N}$ with $u_n^{+}=\max\{0,u_n\}$.	
From this, the equality  
$$
\int_{\R^N}G(u_n)\,dx=1, \quad \forall n \in \mathbb{N}
$$	
leads to
$$
\int_{\R^N}G(u_n^+)\,dx \geq 1, \quad \forall n \in \mathbb{N}.
$$	
Defining the function $h_n:[0,1] \to \R$ by 
$$
h_n(t)=\int_{\R^N}G(t u_n^+)\,dx
$$	
the conditions on $f$ yield that $h$ is continuous with $h_n(1) \geq 1$. Once $u_n^+ \not=0$ for all $n \in  \mathbb{N}$, the condition \eqref{f_{1}} ensures that $h_n(t)<0$ for $t$ close to 0. Thus there is $t_n \in (0,1]$ such that $h_n(t_n)=1$, that is,
$$
\int_{\R^N}G(t_n u_n^+)\,dx = 1, \quad \forall n \in \mathbb{N},
$$	
implying that $t_n u_n^+ \in \mathcal{M}$. On the other hand, we also know that
$$
\int_{\R^N}|(-\Delta)^{\alpha/2}u_n^+|^{2}\,dx  \leq  \int_{\R^N}|(-\Delta)^{\alpha/2}u_n|^{2}\,dx.
$$
Once $t_n \in (0,1]$, the last inequality gives 
$$
D \leq T(t_n u_n^+)\leq T(u_n)=D+o_n(1)
$$
that is,
$$
t_n u_n^{+} \in \mathcal{M} \quad \mbox{and} \quad T(t_n u_n^+) \to D,
$$
showing that $\{t_n u_n^+\}$ is a minimizing sequence for $T$. Thereby, without lost of generality, we can assume that $\{u_n\}$ is a nonnegative sequence. 
	
Moreover, by noticing that
$$
	\int_{\R^N}  |(-\Delta)^{\alpha/2}u_n^{*}|^2 \, dx \leq \int_{\R^N}  |(-\Delta)^{\alpha/2}u_n|^2 \, dx, \quad \forall n \in \mathbb{N} 
	$$
	and
	$$
	\int_{\mathbb{R}^N}G(u_{n}^{*})\, dx= \int_{\mathbb{R}^N}G(u_{n})\, dx, \quad \forall n \in \mathbb{N}
	$$
	where $u_n^{*}$ is the Schwartz symmetrization of $u_n$, 
	any minimizing sequence can  be assumed  radially symmetric, non-negative and decreasing in $r=|x|$.
\end{proof}

In what follow, we will use that the embedding 
\begin{equation} \label{imersao}
H_{rad}^{\alpha}(\R^N) \hookrightarrow L^{p}(\R^N) 
\end{equation}
is compact for all $p \in (2,2_{\alpha}^{*})$, see Lions \cite{Lions} for more details.


\medskip

Due to the boundedness in $H^{\alpha}(\mathbb R^{N})$ of the (non-negative and   radial
symmetric) minimizing sequence $\{u_{n}\}$ (see Lemma \ref{bdedminseq})
we can assume that $\{u_{n}\}$ has a weak limit in $H^{\alpha}(\mathbb R^{N})$
denoted hereafter with $u$. Observe also that,  by the boundedness in $L^{2}(\mathbb R^{N})$  we have the uniform decay $|u_{n}(x)| \leq C |x|^{-N/2}$, see \cite[Lemma 1]{Am}. Therefore, passing to a subsequence, if necessary, we deduce that the weak limit $u$ is non-negative, radially symmetric  and decreasing.
 
It turns out that the weak limit $u$ is a  solution of the minimizing  problem \eqref{eq:D}
we were looking for. Before to see
 this some preliminary lemmas are in order to recover some compactness.

\begin{lemma}\label{nuSD}
Assume that $v_n:=u_n-u\rightharpoonup 0$ in $H^{\alpha}(\mathbb R^{N})$
and  $\displaystyle\int_{\R^N}  |(-\Delta)^{\alpha/2} v_n|^{2} \, dx \to L>0$. Then 
$$
D \geq 2^{-2 \alpha/N}S.
$$ 
\end{lemma}

\begin{proof}
First of all, we recall the limit $T'(u_n) - \lambda_n J'(u_n) \rightarrow 0$ as $n \rightarrow + \infty$ gives
$$
T'(u_n)[u_n]-\lambda_{n} J'(u_n)[u_n]=o_n(1).
$$
Using standard arguments, it is possible to prove that   
$$
T'(u_n)[u_n]-\lambda J'(u_n)[u_n]=T'(v_n)[v_n]-\lambda_{n} J'(v_n)[v_n]+T'(u)[u]-\lambda^{*} J'(u)[u] + o_n(1)
$$
and
$$
T'(u)-\lambda^{*} J'(u)=0 \quad \mbox{in} \quad (H^{\alpha}(\mathbb{R}^N))^{-1}.
$$
Then  $T'(v_n)[v_n]-\lambda_{n} J'(v_n)[v_n]=o_n(1)$, or equivalently,
$$ 
\int_{\R^N}  |(-\Delta)^{\alpha/2} v_n|^{2} \, dx   = 
\lambda_n \,  \int_{\R^N}   f(v_n) v_n \, dx - \lambda_n  \int_{\R^N}  v_n^2 \, dx    +  o_n(1).
$$ 
Using the growth  conditions  on $f$, fixed $q \in (2,2^{*}_{\alpha})$ and given $\varepsilon >0$, there exists $C=C(\varepsilon, q)>0$ such that 
$$
f(t)t \leq \varepsilon t^{2}+C|t|^{q}+(1+\varepsilon)|t|^{2^{*}_\alpha}, \quad \forall t \in \mathbb{R}.
$$
From this,
$$
\int_{\R^N} |(-\Delta)^{\alpha/2}v_n|^2 \, dx \leq  
\lambda_n \left(\varepsilon \int_{\R^N}   v_n^2\, dx  + C  \int_{\R^N} |v_n|^q dx    +(1+\varepsilon) \int_{\R^N}   |v_n|^{2^{*}_{\alpha}} \, dx\right) + o_n(1).
$$
Now, using the definition of $S$, see (\ref{definicaoS}),  we get
\begin{multline}\label{eq:LCS}
\int_{\R^N} |(-\Delta)^{\alpha/2}v_n|^2\,  dx \\ \leq  
\lambda_n \left(\varepsilon \int_{\R^N}   v_n^2 dx  + C  \int_{\R^N} |v_n|^q dx    +(1+\varepsilon) \left( \frac{1}{S} \int_{\R^N}  |(-\Delta)^{\alpha/2} v_n|^{2} \, dx\right)^{2^{*}_{\alpha}/ 2} \right)+ o_n(1).
\end{multline}
Passing to the limit in \eqref{eq:LCS}, 
recalling that $\{v_n\}$ is bounded, 
\begin{equation}\label{eq:L}
\displaystyle\int_{\R^N} |(-\Delta)^{\alpha/2}v_n|^2 \, dx \longrightarrow L
\end{equation}
and that $u_n \to 0$ in $L^{q}(\R^N)$ (see (\ref{imersao})),
we find
$$
L\leq \lambda^*\left( \varepsilon C_1 + (1+\varepsilon)\left(\frac{L}{S}\right)^{2^{*}_{\alpha}/2} \right).
$$
By the arbitrariety of  $\varepsilon$, we derive
$L\leq D \left(L/S \right)^{2^{*}_{\alpha}/2}$, or equivalently, 
\begin{equation}\label{eq:SDL}
S^{2^{*}_{\alpha}/2}\leq  D L^{2 \alpha /(N-2 \alpha)}.
\end{equation}
On the other hand \eqref{eq:L} implies that 
$L = 2D -\displaystyle\int_{\R^N} |(-\Delta)^{\alpha/2}u|^2 dx \leq 2D$. Hence \eqref{eq:SDL} becomes
$$
S^{2^{*}_{\alpha}/2}\leq  2^{2 \alpha /(N-2 \alpha)} D^{2^{*}_{\alpha}/2},
\quad \text{ i.e. } \quad D \geq 2^{-2 \alpha/N}S
$$
and the proof is finished.
\end{proof}

In the next result the condition $\tau>\tau^{*}$ given in \eqref{f_{4}} plays a crucial role.
\begin{lemma}\label{lambdabbound}
It holds
$$
b <  \frac {\alpha}{ N} \left(\frac{N-2\alpha}{2N} \right)^{(N-2\alpha)/2\alpha} 2^{(N-2\alpha)/2\alpha} S^{N/2\alpha}.
$$ 
 
\end{lemma}

\begin{proof}
Take $\varphi \in H^\alpha(\R^N)$ such that $\| \varphi \| =1$ and $| \varphi |_q^2 = C_q^{-1}$. 
From definition of $ b = \inf_{\gamma \in \Gamma} \max_{t \in [0, 1]} I(\gamma(t))$ and  \eqref{f_{4}}
\begin{eqnarray*}
b \leq \max_{ t \geq 0} I(t \varphi) & \leq & \max_{t \geq 0} \left\{ \frac{t^2}{2} - \tau \frac{t^q}{q}  \int_{\R^N}   |\varphi|^q \, dx  \right\} \\
&   = &  \max_{t \geq 0} \left\{ \frac{t^2}{2} - \tau \frac{t^q}{q}  C_q^{-q/2}  \right\} \\
&   = & \frac{q-2}{2q} \,  \frac{C_q^{q/(q-2)}}{\tau^{2/(q-2)}}.
\end{eqnarray*}
This gives (by the definition of  $\tau^{*}$)  exactly the conclusion.
\end{proof}

\begin{lemma}\label{weaklim}
If  $u_n \rightharpoonup u$ in $H^{\alpha}(\R)$, then  $u_n \to u$ in $\mathcal D^{\alpha,2}(\mathbb{R}^N)$. 
In particular,  $u_n \to u$  in  $L^{2^{*}_{\alpha}}(\mathbb{R}^N)$. 
\end{lemma}
\begin{proof}
Of course  $v_n=u_n -u \rightharpoonup 0$ in $H^{\alpha}(\R)$.
Suppose by contradiction that $u_n \not\to u $ in $\mathcal D^{\alpha,2}(\mathbb{R}^N)$. Thereby, 
$\displaystyle\int_{\R^N}  |(-\Delta)^{\alpha/2} v_n|^{2} \, dx \to L>0$ for some subsequence.Then, by Lemma \ref{nuSD},
\begin{equation}\label{eq:DS}
D \geq 2^{- 2\alpha/N} S.
\end{equation}
On the other hand, from Lemma \ref{blowbound}  
$$\frac {\alpha}{N} \left(\frac{N-2\alpha}{2N} \right)^{(N-2\alpha)/2\alpha} (2D)^{N/2\alpha} \leq b,$$
from which, using \eqref{eq:DS}, it follows that  
$$ 
\frac {\alpha}{N} \left(\frac{N-2\alpha}{2N} \right)^{(N-2\alpha)/2\alpha} 2^{(N-2\alpha )/2\alpha} S^{N/2\alpha} \leq b.
$$ 
This contradicts  Lemma \ref{lambdabbound} and finishes the proof. 
\end{proof}

\medskip


\subsection{Proof of Theorem \ref{groundN3}}\label{subsec:th1}

\noindent At this point we wish to show that $D$ is attained by $u$, where $u$ is the weak limit of $\{ u_n \}$. First of all, we know that
\begin{equation}\label{TDN}
T(u) = \frac 12\int_{\R^N}  |(-\Delta)^{\alpha/2} u|^2  \, dx \leq \liminf_{n \rightarrow + \infty} \frac 12 \int_{\R^N}  
|(-\Delta)^{\alpha/2} u_n|^2  \, dx =D
\end{equation}
so we just need to prove that $u\in \mathcal M$.

By \cite[Lemma 1]{Am}, there is $R>0$ such that 
$$
\frac{1}{2}u_n^{2}-F(u_n) \geq 0 \quad \forall n \in \mathbb{N} \quad \mbox{in } \mathbb R^{N} \setminus  B_{R},
$$
$B_{R}$ being the ball of radius $R$ centered in $0.$
Since
$$
\int_{B_R}F(u_n)\,dx=\frac{1}{2}\int_{B_R}u_n^{2}\,dx+\int_{\mathbb{R}^N \setminus B_R}\left(\frac{1}{2}u_n^{2}-F(u_n)\right)\,dx +1
$$
and  $u_n \to u $ in $L^{2^{*}_{\alpha}}(B_R)$, the above information together with the Fatous' Lemma gives
$$
\displaystyle\int_{B_R}F(u)\, dx \geq 
\frac{1}{2}\int_{B_R}u^{2}\,dx+\int_{\mathbb{R}^N \setminus B_R}\left(\frac{1}{2}u^{2}-F(u)\right)\,dx +1
$$
which leads to
$$ 
\displaystyle\int_{\R^N}  G(u) \, dx  \geq 1.
$$
Suppose by contradiction that
$$ 
\int_{\R^N}  G(u) \, dx  > 1
$$
and define $h: [0, 1] \rightarrow \R$ by $h(t) = \int_{\R^N}  G(tu) \, dx$. The growth conditions on $f$  ensure that $h(t) < 0$ for $t$ close to $0$ and $h(1) = \int_{\R^N}  G(u) \, dx > 1$. 
 Then, by the continuity of $h$, there exists $t_0 \in (0, 1)$ such that $h(t_0)=1$. Then, 
$$\int_{\R^N}  G(t_0 u) \, dx = 1      \Longleftrightarrow t_0 u \in \mathcal{M}.$$
Consequently, by \eqref{TDN}
$$D \leq T(t_0 u) = \frac{t_0^2}{2} \int_{\R^N} |(-\Delta)^{\alpha/2} u|^2 \, dx = t_0^2 \, T(u) \leq t_0^2 \, D < D$$
which is absurd. Thus $ \int_{\R^N}  G(u) \, dx  =1 $, i.e. $ u\in \mathcal M$.
The fact that the solution $u$ of the minimizing problem
 gives rise to a ground state solution, follows by standard arguments; indeed, 
since $ u$ is a solution of the minimizing problem \eqref{eq:D}, i.e. $D =T(u)= \inf_{w\in \mathcal M}T(w)$, then there exists an  associated Lagrange multiplier $\lambda$ such that, in a weak sense,
$$
(-\Delta)^{\alpha}  u =  \lambda g( u).
$$
Now by  testing the previous equation on the same minimizer $u$, we deduce that
$$2T(u)=\lambda\int_{\mathbb R^{N}}g(u)u dx = \lambda J'(u)[u] \geq 2\lambda $$
so that it has to be, by Lemma \ref{Dposit}, $T(u)\geq\lambda>0$.
Setting $u_{\sigma}(x):=u(\sigma x)$ for $\sigma>0$, we easily see that
$$
(-\Delta)^{\alpha} u_{\sigma} = \lambda \sigma^{2\alpha}g(u_{\sigma}).
$$
Choosing $\sigma=\lambda^{1/2\alpha}$ we obtain a solution of \eqref{eq:P}.
 Arguing as in  \cite[Theorem 3]{BL}, $u_\sigma$ is a ground state solution.

%

\section{The case $N=1$ and $\alpha= 1/2$}\label{tools1}

\subsection{The variational framework}
As for the previous case, let us consider the set of nontrivial solutions of
\eqref{eq:P}, namely
$$
\Sigma = \{  u \in H^{1/2}(\R) \setminus \{0 \}: I'(u) =0 \}, 
$$  
and let 
$$  m = \inf_{u \in \Sigma} I(u).$$ 

Denoting with   $G(u) =  F(u)  -\displaystyle\frac{u^2}{2}$, the primitive of
$g(u) = f(u)  -u$,  we introduce  the set 
\begin{equation}\label{eq:MN=1}
\mathcal{M}=  \left\{ u \in H^{1/2}(\R) \setminus \{0 \}: \int_{\R} G(u) \, dx =0 \right \} ,
\end{equation} 
and 
\begin{equation}\label{DefD}
T(u) =\frac 12 \int_{\R}  |(-\Delta)^{1/4}u|^2  dx,\qquad D =  \inf_{u\in \mathcal M}  T(u). 
\end{equation}
it is, again as before,
$$ 2 D = \inf_{u \in \mathcal{M}} \Big\{  \int_{\R}  |(-\Delta)^{1/4}u|^2  dx\Big\}. 
$$ 
We point out here that, since
 we will deal with minimizing sequences $\{u_{n}\}$
 for the minimization problem \eqref{DefD}, as in the previous Section we suppose that  $u_{n}$ is non-negative and radially symmetric. Moreover, we again define the min-max level associated to the functional $I$
\begin{equation}\label{eq:bN=1}
b = \inf_{\gamma \in \Gamma} \max_{t \in [0, 1]} I(\gamma(t))
\end{equation}
where 
$$ 
\Gamma = \{\gamma \in C\left([0, 1], H^{\alpha}(\R) \right): \gamma(0) =0 \ \hbox{ and } \ I(\gamma(1)) < 0\}.
$$



\subsection{Some preliminary stuff}

Let us start with the following important result due to T.~Ozawa \cite{Ozawa} 

\begin{theorem} \label{th:Ozawa}
	There exists $0 < \omega \leq \pi$ such that, for all $r \in (0, \omega)$, there exists $H_{r}>0$ satisfying
	\begin{equation*}\label{I1}
	\int_{\mathbb{R}}( e^{r u^2 }-1) \,d x \leq H_{r}|u|^{2}_{2},
	\end{equation*}
	for all $u\in H^{1/2}(\mathbb{R})$ with $|(-\Delta)^{1/4} u|^2_{2}\leq 1$.
\end{theorem}


At this point we establish some preliminary  results which will be useful in order to prove Theorem \ref{groundN1}.

\begin{lemma}\label{Mmanif1}
The set $\mathcal{M}$ defined in \eqref{eq:MN=1} is  not empty and  a $C^1$ manifold. 
\end{lemma}

\begin{proof}
Consider $w \in C^{\infty}_{0}(\R)$ with $w(x)>0$ and define a function
$$
h(t)
=\displaystyle\int_{\R}G(tw)dx=\displaystyle\int_{\R}F(tw) dx-\frac {t^{2}}{2} \displaystyle\int_{\R}w^{2}dx.
$$
From \eqref{f_{1}} for $t>0$ small we have
$$
h(t)\leq \frac{\varepsilon -1}{2}t^{2}\displaystyle\int_{\R}w^{2}dx.
$$
For $\varepsilon <1$, we get $h(t)<0$ for $t>0$ small.

Now using \eqref{f_{4}} we obtain
$$
h(t) \geq \tau \frac{t^{q}}{q}\displaystyle\int_{\R}w^{q}dx-\frac {t^{2}}{2} \displaystyle\int_{\R}w^{2}dx
\quad {and }\quad h'(t) \geq \lambda t^{q-1}\displaystyle\int_{\R}w^{q}dx- t \displaystyle\int_{\R}w^{2}dx.
$$
Then, $h(t)>0$ for $t>0$ large
and $h'(t)>0$ for $t>0$ large. Then there is a  $\overline{t}>0$ such that
$$
\displaystyle\int_{\R}G(\overline{t}w)\,dx
=h(\overline{t})=0.
$$

Now we prove that $\mathcal{M}$ is a manifold. Indeed, if  $w\in \mathcal{M}$, then  $w \not=0$. Then, from \eqref{f_{1}} and  the fact that
$\lim_{|x|\to \infty}w(x)=0 $, there exists $ x_0 \in \R$ such that $g(w(x_0))<0$. Thereby, by continuity,  there is an open interval $B_\delta(x_0)$ such that 
$$
g(w(x))<0, \quad  \forall x \in B_\delta(x_0).
$$
As a consequence we can always find a $\phi \in C_0^{\infty}(\R) \subset H^{1/2}(\R)$ such that $J'(w)[\phi]=\int_{\R}g(w)\phi \,dx < 0$, showing that  $J'(w) \not=0$.	
\end{proof}

\begin{lemma}\label{ConvergenceofF}
Assume that $f$ satisfies \eqref{f_{1}},\eqref{f_{4}} and  \eqref{f_{5}}. Let $\{v_n\}\subset H^{1}(\mathbb R)$ be a sequence of radial functions such that 
$$
v_n \rightharpoonup v \ \ \mbox{in} \ \ H^{1/2}(\R)
$$
and
$$
\displaystyle\sup_{n}|(-\Delta)^{1/4} u|^{2}_{2}=\rho <1 \ \ \mbox{and} \ \ \displaystyle\sup_{n}|v_n|^{2}_{2}=M <\infty.
$$
Then,
$$
\displaystyle\int_{\R}F(v_n)dx\longrightarrow \displaystyle\int_{\R}F(v)dx.
$$
\end{lemma}
\begin{proof}
Without loss of generality, we can assume that there is $v \in H^{1/2}(\R)$, radial, such that 
$$
v_n \rightharpoonup v \ \ \mbox{in} \ \ H^{1/2}(\R), \ \ v_n(x) \to v(x) \ \ \mbox{a.e in} \ \ \R \ \ \mbox{and} \ \ \displaystyle\lim_{|x|\to +\infty} v_n(x)=0, \ \ \mbox{uniformly in} \ \ n.
$$
Using the Theorem \ref{th:Ozawa}, we know that for each $m\in(0,1)$ and $M>0$, there exists $C(m,M)>0$ such that
$$
\displaystyle\sup_{u\in \texttt{B}}\int_{\R}(e^{w u^2}-1) \, d x  \leq C(m,M),
$$
where 
$$
\texttt{B}=\biggl\{ u \in  H^{1/2}(\R) : \ |(-\Delta)^{1/4} u|^{2}_{2}\leq m \ \ \mbox{and} \ \ |u|^{2}_{2}\leq M \biggl\}.
$$
Now, choose $\varepsilon>0$ small enough such that $m=\frac{\rho}{(1-\epsilon)^{2}} \in (0,1)$ and set $t=\frac{\omega}{(1-\epsilon)^{2}} > w \geq \beta_0$. Then,
$$
\displaystyle\int_{\R}(e^{t v_{n}^2}-1) dx=\displaystyle\int_{\R}(e^{t(1-\epsilon)^{2} (\frac{v_{n}}{1-\epsilon})^2}-1)dx 
=  \displaystyle\int_{\R}(e^{\omega (\frac{v_{n}}{1-\epsilon})^2}-1)dx.
$$
Since $v_{n}\in \texttt{B}$ we have
$$
\displaystyle\int_{\R}(e^{t v_{n}^2}-1)dx \leq \displaystyle\sup_{u\in \texttt{B}}\int_{\R}(e^{\omega u^2}-1) \,d x  \leq C(m,M).
$$
Now, setting $P(s)=F(s)$ and $Q(s)=e^{ts^{2}}-1$, from \eqref{f_{1}}, \eqref{f_{4}} and the last inequality, we get
$$
\displaystyle\lim_{s\to 0}\frac{P(s)}{Q(s)}=\displaystyle\lim_{s\to+ \infty}\frac{P(s)}{Q(s)}=0,
$$
$$
\displaystyle\sup_{n\to+ \infty}\int_{\R}Q(v_n) dx<\infty
$$
and
$$
P(v_{n}(x))\longrightarrow P(v(x)) \ \ \mbox{a.e in} \ \ \R.
$$
Consequently the hypotheses of the Compactness Lemma of Strauss \cite[Theorem A.I]{BL} are fulfilled. Hence $P(v_{n})$ converges to 
$P(v)$ in $L^{1}(\R)$, and then 
$$
\displaystyle\int_{\R}F(v_n)dx\longrightarrow \displaystyle\int_{\R}F(v)dx
$$
concluding the proof.
\end{proof}

The relation between the ground state level and the minimax level defined in 
\eqref{eq:bN=1} is given in the following

\begin{lemma}\label{Dandb}
The numbers $D$ and $b$ satisfy the inequality $D\leq b$.
\end{lemma}
\begin{proof}
Arguing as in Lemma \ref{Mmanif1}, given $v\in H^{1/2}(\R)$ with $v^{+}=\max\{v,0\}\neq 0$, there is $t_{0}>0$ such that $t_0 v^{+} \in \cal{M}$. Then, 
$$
D\leq \frac{t_{0}^{2}}{ 2} \int_{\R}  |(-\Delta)^{1/4}  v^{+}|^2  \, dx = I(t_{0}v^{+})\leq \displaystyle\max_{t\geq 0}I(tv^{+}).
$$
On the other hand, since $f(s)=0$ for $s\leq 0$, if $v\in H^{1/2}(\R)$, $v\neq 0$ with $v^{+}=0$, \linebreak then $\max_{t\geq 0}I(tv)=\infty$. Hence in any case $D\leq b$.
\end{proof}

\begin{lemma}\label{Dposit}
The number $D$ given by \eqref{DefD} is positive, namely, $ D > 0$.
\end{lemma}

\begin{proof}
By definition $D\geq 0$. Assume by contradiction that $D=0$ an let $\{u_n\}$ be a 
(non-negative and radial) minimizing
sequence in $H^{1/2}(\R)$ for $T$, that is, 
$$
\displaystyle\int_{\R} |(-\Delta)^{1/4}u_n|^2 dx \to 0 \ \ \mbox{and} \ \ \displaystyle\int_{\R} G(u_n) \, dx =0. 
$$
For each $\mu_n>0$, the function  $v_n(x):=u_n( x / \mu_n)$ satisfies
$$
\displaystyle\int_{\R} |(-\Delta)^{1/4}v_n|^2dx= 
\displaystyle\int_{\R} |(-\Delta)^{1/4}u_n|^2dx \ \ \mbox{and} \ \ \displaystyle\int_{\R} G(v_n) \, dx =0.
$$
Since 
$$
\displaystyle\int_{\R} v_n^{2}dx= \mu_{n}^{2}\displaystyle\int_{\R} u_n^{2}dx,
$$
we choose $\mu_{n}^{2}=|u_{n}|_{2}^{-2}
$
 to obtain 
$$
\displaystyle\int_{\R} |(-\Delta)^{1/4}v_n|^2 dx\to 0, \ \ \displaystyle\int_{\R} v_n^{2}dx=1 \ \ \mbox{and} \ \ \displaystyle\int_{\R} G(v_n) \, dx =0.
$$
 and we can assume that
there exists $ v \in H^{1/2}(\R)$, radial, such that  $v_n \rightharpoonup v$ in $H^{1/2}(\R)$. From Lemma \ref{ConvergenceofF} we get
$$
\displaystyle\int_{\R}F(v_n)dx\longrightarrow \displaystyle\int_{\R}F(v)dx.
$$
Note that $\int_{\R} G(v_n) \, dx =0$ implies $\int_{\R} F(v_n) \, dx =\frac 12$ and $\int_{\R} F(v) \, dx =\frac 12$. Then $v \neq 0$. But
$$
\displaystyle\int_{\R} |(-\Delta)^{1/4}v|^2dx \leq \liminf_{n\to+\infty} \displaystyle\int_{\R} |(-\Delta)^{1/4}v_n|^2dx 
\longrightarrow 0,
$$
implies $v=0$ which is an absurd.
\end{proof}

\begin{lemma}\label{lambdabbound1}
We have
$ b <   1/2.$ 
 \end{lemma}
\begin{proof}
It is sufficient to repeat the same argument of the Lemma \ref{lambdabbound}, 
recalling that now by \eqref{f_{4}} it is 
$$
\tau^* = \left( \frac{q-2}{q} \right)^{(q-2)/2} C_q^{q/2}
$$
concluding the proof.
\end{proof}

\subsection{Proof of Theorem \ref{groundN1}}\label{sectmainres1}

At this point we will show that $D$ is attained by $u$, where $u$ is the weak limit of $\{ u_n \}$. Indeed, since $u_n \rightharpoonup u$ in $H^{1/2}(\R)$ we have 
\begin{equation}\label{TD}
T(u) = \frac 12\int_{\R}  |(-\Delta)^{1/4} u|^2  \, dx \leq \liminf_{n \rightarrow + \infty} \frac 12 \int_{\R}  
|(-\Delta)^{1/4} u_n|^2  \, dx =D.
\end{equation}
Moreover, by Lemma \ref{ConvergenceofF} we have 
$$
\displaystyle\int_{\mathbb{R}}F(u) dx= \displaystyle\liminf_{n\to\infty}\displaystyle\int_{\mathbb{R}}F(u_n) dx\geq \frac 12 \displaystyle\int_{\mathbb{R}}u^{2} dx
$$
leading to  
$$ 
\displaystyle\int_{\R}  G(u) \, dx  \geq 0. 
$$
As in the previous case $N\geq2$, we just need to prove that $u\in \mathcal M$, i.e. $ \int_{\R^N}  G(u) \, dx  = 0$.

We again argue  by contradiction by supposing that 
$$ 
\int_{\R}  G(u) \, dx  > 0.
$$
As in the previous section, we set $h: [0, 1] \rightarrow \R$ by $h(t) = \int_{\R}  G(tu) \, dx$. Using the growth condition of $f$ we have $h(t) < 0$ for $t$ close to $0$ and $h(1) = \int_{\R}  G(u) \, dx > 0$. Then, by the continuity of $h$, there exists $t_0 \in (0, 1)$ such that $h(t_0)=0$, that is $t_{0}u\in \mathcal M$.
Consequently, by \eqref{TD}
$$D \leq T(t_0 u) = \frac{t_0^2}{2} \int_{\R^N} |(-\Delta)^{\alpha/2} u|^2 \, dx = t_0^2 \, T(u) \leq t_0^2 \, D < D$$
which is absurd. 

As for the case $N\geq2$, one show that the minimizer $u$ of \eqref{DefD} gives rise to a ground state solution of 
\eqref{eq:P}.

%


\begin{thebibliography}{1} 

\bibitem{AlvesSoutoMontenegro}
{C. O. Alves, M. Montenegro and M.A. Souto, }
{\em Existence of a ground state solution for anonlinear scalar field equation with critical growth}, Calc. Var. and PDEs 43 (2012), 537-554.

\bibitem{Am} {V. Ambrosio}, {\em Zero mass case for a fractional Beresticky-Lions type results}, preprint 


\bibitem{barrios} B.\ Barrios, E.\ Colorado, A.\ de Pablo and U.\ S\'anchez, 
On some critical problems for the fractional Laplacian operator, 
J. Differential  Equations {\bf 252} (2012), 613--6162.



\bibitem{BL} {H. Berestycki and P.L. Lions, } {\em Nonlinear Scalar Field, }
Arch. Rational Mech. Anal. 82 (1983), 313-345

\bibitem{BGK} {H. Berestycki, T. Gallouet and O. Kavian,} {\em Equations de Champs scalaires euclidiens non linéaires dans le plan. }, C. R. Acad. Sci. Paris Ser. I Math. 297, 307–310 (1984)

\bibitem{Cao} {D.M. Cao,} {\em  Nonlinear solutions of semilinear elliptic equations with critical exponent in $\mathbb{R}^2$}. Comm. Part. Diff. Equat. 17, 407–435 (1992)


\bibitem{ChangWang} {X. Chang and Z-Q Wang}, {\em Ground state of scalar field equations involving a fractional Laplacian with general nonlinearity}, {Nonlinearity 26 (2013) 479-494.}

\bibitem{DPV}{ E. Di Nezza, G. Palatucci and E. Valdinoci, }
{\em Hitchhiker's guide to the fractional Sobolev spaces, }
Bull. Sci. Math. 136 (2012), 512--573.


\bibitem{FL}{R. Frank and E. Lenzmann, }
{\em Uniqueness of non-linear ground states for fractional Laplacians in $\mathbb R$,} Acta Math., 210 (2013), 261-318.


\bibitem{FQT} P. Felmer, A Quass and J. Tan, {\em
Positive solutions of nonlinear Schr\"odinger equation with the fractional Laplacian, }
Proc. Roy. Soc. Edinburgh A {\bf 142} (2012), 1237–-1262.


\bibitem{IS} A.\ Iannizzotto and  M.\ Squassina,  {\em 
$1/2$-laplacian problems with exponential nonlinearity, }
J. Math. Anal. Appl. {\bf 414} (2014), 372--385.


\bibitem{JJTan} A. Jeanjean and  K. Tanaka, 
{\em A remark on least energy solutions in $\R^N$}, Proc. Amer. Math. Soc. \textbf{131} (2002),  2399-2408.


\bibitem{LilianeRquelSquassina} Raquel Lehrer, Liliane A. Maia and Marco Squassina, 
{\em Asymptotically linear fractional Schr\"odinger equations, } 
Complex Variables and Elliptic Equations: An International Journal, 60:4, (2015), 529-558.

\bibitem{lionsI}  P.L. Lions, {\em The concentration-compactness principle in the calculus of variations. The locally compact case. Part I.},  
Ann. Inst. Henri Poincar\'e, Anal. Non Lin\'eaire 1 (1984),  109-145.

\bibitem{Lions} P. L. Lions, {\em Sym\'etrie et compacit\'e dans les espaces de Sobolev,} J. Funct. Analysis, 49 (1982), 315–334. 

\bibitem{PalatucciPisante}  G. Palatucci and A. Pisante,
 {\em Improved Sobolev embeddings, profile decomposition, and concentration-compactness for fractional Sobolev 
 spaces}, Calc. Var. 50 (2014), 799-829.


\bibitem{Ozawa}  T.Ozawa, 
{\em On critical cases of Sobolev's inequalities, }
J. Funct. Anal. {127} (1995), 259--269. 



\bibitem{ZZ} J. Zhang and W. Zou, {\em A Berestycki-Lions theorem revisited,} Comm. Contemp. Math. 14 (2012), 1250033-1.

\bibitem{ZOS} J. J. Zhang, J. M. do \'O and M. Squassina,{\em Fractional  Schr\"{o}dinger-Poisson systems with a general subcritical or critical nonlinearity}, Advanced Nonlinear Studies 16 (2016), 15-30.


\end{thebibliography}
\end{document}